\documentclass[a4paper]{article}

\usepackage[UKenglish]{babel}
\usepackage[UKenglish]{isodate}
\usepackage{amsfonts}
\usepackage{amssymb}
\usepackage{amsthm}
\usepackage{graphicx}
\usepackage{amsmath}
\usepackage{bbm}
\usepackage[a4paper]{geometry}
\usepackage{hyperref}

\newtheorem{theorem}{Theorem}

\newtheorem{corollary}{Corollary}

\newtheorem{lemma}{Lemma}

\newtheorem{proposition}{Proposition}

\usepackage[authoryear, sort&compress]{natbib}
\def\lessim{\ \lower4pt\hbox{$\buildrel{\displaystyle <}\over\sim$}\ }

\newcommand{\1}{\mathbbm{1}}

\cleanlookdateon

\begin{document}

\title{Uniform central limit theorems for the\\Grenander 
estimator}

\author{Jakob S\"ohl \footnote{The author thanks Evarist Gin\'e and Richard 
Nickl for
providing him with a preliminary version of their forthcoming book. This paper
is based on a chapter from their book and extends the results therein. The 
author is grateful to
Richard Nickl for helpful discussions on the topic.}\\
\\
\textit{University of Cambridge \footnote{Statistical Laboratory, Department of
Pure Mathematics and Mathematical Statistics, University of Cambridge, CB3 0WB,
Cambridge, UK. Email: j.soehl@statslab.cam.ac.uk}}}

\date{26 June 2015}

\maketitle

\begin{abstract}

We consider the Grenander estimator that is the maximum likelihood estimator 
for non-increasing densities. We prove uniform central limit theorems for
certain subclasses of bounded variation functions and for H\"older balls of
smoothness $s>1/2$. We do not assume that the density is differentiable or 
continuous.
The proof can be seen as an adaptation of the method for the 
parametric maximum likelihood estimator to the nonparametric setting.
Since nonparametric maximum likelihood estimators lie on the boundary, the
derivative of the likelihood cannot be expected to equal zero as in the
parametric case. 
Nevertheless,
our proofs rely on the fact that the derivative of the likelihood can be
shown to be small at the maximum likelihood estimator. 

\medskip


\end{abstract}

\section{Introduction}

A fundamental approach to statistical estimation is finding the probability
measure which renders the observation most likely. This principle of maximum
likelihood estimation has proved very successful in parametric estimation but
leads to difficulties in nonparametric problems since the likelihood is
typically unbounded so that no maximum is attained.
However, in nonparametric problems with shape constraints the maximum
likelihood
estimator is often well-defined and thus the maximum likelihood approach can be
extended to these situations. Examples include non-increasing, $k$-monotone,
convex, concave and log-concave functions.

The classical
parametric maximum likelihood theory is based on the estimator $\hat\theta_n$
being in the interior of the parameter space and on the resulting fact that the
derivative of the likelihood vanishes at $\hat\theta_n$. The theory of
nonparametric maximum likelihood estimation is quite different from the
parametric theory since the estimator lies on the boundary of the parameter
space and thus in general the derivative of the likelihood will not be
zero. 
But in some nonparametric situations the derivative of the likelihood can be 
shown to be sufficiently small, thus enabling a proof strategy paralleling 
the one in the classical parametric theory.
\cite{Nickl2007} considered the nonparametric maximum likelihood estimator for 
estimating a 
density in a Sobolev ball and proved uniform central limit theorems using this 
approach.

We pursue this method of proof in the problem of estimating a non-increasing 
density $p_0$. 
The maximum likelihood estimator $\hat p_n$ is called the
Grenander estimator in this situation since it was first derived by 
\cite{Grenander1956}.
It is well known to be the left-derivative of the least concave
majorant of the empirical distribution function. 
The main results will be
uniform central limit theorems for the Grenander estimator that in particular 
imply for functions~$f$
\begin{align*}
 \sqrt n \int_0^\infty (\hat p_n(x)-p_0(x))f(x)dx \to^d N(0,
\|f-P_0f\|_{L^2(P_0)}^2),
\end{align*}
where $P_0$ is the probability measure of the non-increasing density $p_0$ and 
$P_0 f\equiv \int_0^\infty f(x)dP_0(x)$. 
We do not assume the density $p_0$ to be differentiable or continuous.
Our results are uniform for $f$ varying in a 
class of functions
and we cover two different types of classes. 
The first type is a subclass of the bounded variation functions and for 
each point of 
discontinuity $t$ of $p_0$ the indicator function $\1_{[0,t]}$ is contained in 
such 
a class.
The 
second type of classes is given by balls in H\"older spaces $C^s$ of order 
$s>1/2$. 

Under a strict curvature condition and for continuously differentiable $p_0$,
\cite{KieferWolfowitz1976} proved that the difference between the 
distribution function of the Grenander
estimator and the empirical distribution
function in supremum norm is with probability one of order
$n^{-2/3}\log(n)$. On the one hand this means that the two distribution 
functions are close and the distribution function of the Grenander estimator 
does essentially not improve on the empirical distribution function. On the 
other hand it shows that the distribution function of the 
Grenander estimator enjoys many optimality properties of the empirical 
distribution function. The Kiefer--Wolfowitz theorem can be used to prove 
that the distribution function of the Grenander 
estimator is an asymptotically minimax estimator for concave 
distribution functions. 
It further implies a uniform central limit theorem for the Grenander estimator 
over the class of all indicator functions $\1_{[0,t]}$, $t\ge0$. 
The Kiefer--Wolfowitz theorem was used by
\cite{SenBanerjeeWoodroofe2010} to study 
consistency and inconsistency of bootstrap methods when estimating a 
non-increasing density.

Results similar to the Kiefer--Wolfowitz theorem hold under other shape
constraints as well. In addition to giving an updated proof of the 
Kiefer--Wolfowitz theorem,
\cite{BalabdaouiWellner2007} showed such a theorem in the case 
where the density
is assumed
to be convex decreasing and where the maximum likelihood
estimator is replaced by the least squares estimator.
\cite{DuembgenRufibach2009} derived the rate of estimation for log-concave
densities in supremum norm and showed that the difference between empirical and
estimated distribution function is $o(n^{-1/2})$.
\cite{DurotLopuhaae2014} showed a general Kiefer--Wolfowitz type theorem which 
covers the estimation of monotone regression curves, monotone densities and 
monotone failure rates.

\cite{Jankowski2014} studied the local convergence rates of the Grenander
estimator in situations where it is misspecified and derived the asymptotic
distribution of linear functionals under possible misspecification. She 
distinguishes between curved and flat parts regarding the density~$p_0$. On the 
curved parts our results have the advantage that we do not need to assume 
that~$p_0$ and~$f$ are differentiable, whereas on the flat parts 
\cite{Jankowski2014} is more general by treating $L^p$ functions $f$.
We will discuss this in detail in Section~\ref{sec:discussion}.
Beyond asymptotic normality for
single functionals our work includes a uniformity in the underlying
functional, which is important to apply the results by \cite{Nickl2009} 
concerning convolutions of density estimators and to show the ``plug-in 
property'' introduced by
\cite{BickelRitov2003}.
Estimators with this property are 
rate optimal density estimators that simultaneously lead to the efficient 
estimation of functionals with uniform convergence over a class of functionals.
We will elaborate on these applications in Section~\ref{sec:discussion}.
\cite{Nickl2007} discusses applications of uniform 
central limit theorems in the context of the maximum likelihood estimator over
a Sobolev ball.
Uniform central limit theorems were also shown for kernel density estimators and
for wavelet density estimators by \citet{GineNickl2008,GineNickl2009}.

The method of proof presented in this paper is of general nature and its 
relevance extends to maximum likelihood estimators in other problems with shape 
constraints. For example one could consider the estimation of classes of convex 
decreasing densities or, more general, of $k$-monotone densities or of 
other shape constrained densities as long as the classes are convex and closed 
with respect to the supremum norm.

This paper is organised as follows. In Section~\ref{sec:results} we state the 
uniform
central limit theorems for a subclass of the bounded variation functions and
for H\"older balls. 
Section~\ref{sec:discussion} provides a discussion and applications of our 
results. In Section~\ref{sec:deriv} we explain the
general approach. In Section~\ref{sec:bounds} we derive upper and lower bounds
in probability for the Grenander estimator and recall the $L^2$-convergence
rate. Section~\ref{sec:proofs} develops the approach further
and contains the proofs of the main results.

\section{Main results}\label{sec:results}

Let $X_1, \dots, X_n$ be i.i.d.~on $[0,\infty)$ with law $P_0$ and distribution
function $F_0(x)=\int_0^x dP_0, x \in [0,\infty)$. In order to state the main 
results we 
introduce some notation.
We define the empirical measure $P_n
= n^{-1} \sum_{i=1}^n \delta_{X_i}$, the empirical cumulative distribution
function $F_n(x) = \int_0^x dP_n, x \in [0,\infty)$ and 
the log-likelihood function
\begin{align} \label{LIKlog}
\ell_n(p)=\frac{1}{n}\sum_{i=1}^n \log p(X_i).
\end{align}
Under the assumption $E_{P_0}|\log p(X)|<\infty$ for all $p \in \mathcal P$, we
can define the
limiting log-likelihood function
\begin{equation}  \label{LIKllog}
\ell(p) = \int_0^\infty \log p(x) dP_0(x).
\end{equation}
If $P$ is known to have a
monotone decreasing density $p$ then the associated maximum likelihood estimator
$\hat p_n $ maximises the log-likelihood function
$\ell_n(p)$ over $$\mathcal P
\equiv \mathcal P^{mon} = \left \{p:[0,\infty) \to [0, \infty), \int_0^\infty 
p(x)dx 
=1,
\text{ $p$ is non-increasing} \right\},$$ that is,
\begin{equation} \label{LIK0}
\max_{p \in \mathcal P^{mon}} \ell_n(p) = \ell_n(\hat p_n).
\end{equation}
The maximum likelihood estimator $\hat p_n$ is known to be the left-derivative
of the least concave majorant $\hat F_n$ of the empirical distribution function
$F_n$.
Let $\hat P_n$ be the probability measure corresponding to the 
density $\hat p_n$.
For a set $T$ let $\ell^\infty(T)$ denote the space of bounded
real-valued functions on $T$ with the usual supremum norm $\|\cdot\|_\infty$.
Throughout we will denote by $\to^d$ the convergence in distribution as in
Chapter~1 in \cite{vanderVaartWellner1996}.
The $P_0$-Brownian bridge $G_{P_0}$ is defined as tight Gaussian random
variable arising from the centred Gaussian process with covariance
\begin{align*}
 \mathbb{E}[G_{P_0}(f)G_{P_0}(g)]=P_0(fg)-P_0fP_0g,
\end{align*}
where $P_0f\equiv\int_0^\infty f(x) d P_0(x)$.

The first main result is a uniform central limit theorem for a subclass of the
bounded variation functions. We start with the general result in 
Theorem~\ref{thm:bv} and consider its consequences in Corollary~\ref{cor:jump} 
and Theorem~\ref{LIKkw2}. 
Let $f,p_0\in L^1([0,\infty))$ and assume that the weak derivatives of 
$f|_{(0,\infty)}$ and $p_0|_{(0,\infty)}$ in the sense of regular Borel signed 
measures exist and denote them by $Df$ and $Dp_0$, respectively; cf., e.g., 
p.~42 in \cite{Ziemer1989}.
We define $BV[0,\infty)\equiv\{f\in 
L^1([0,\infty)):\|f\|_1+|Df|(0,\infty)<\infty\}$, where 
$|Df|$ is total variation of the signed measure $Df$.
In the following theorem it will be important that $Df$ is absolutely 
continuous with respect to $Dp_0$ since we want to ensure that the 
perturbations $p_0\pm \eta f$ with $|\eta|$ small (or slightly modified 
perturbations) are decreasing functions. To this end we denote 
the Radon--Nikodym derivative of $Df$ with respect to $Dp_0$ by $Df/Dp_0$ and 
assume that its essential 
supremum with respect to $Dp_0$, denoted by
$\|Df/Dp_0\|_{\infty,Dp_0}$, is bounded. We will consider decreasing 
densities~$p_0$ with bounded support~$S_0$. Then we can write 
$S_0=[0,\alpha_1]$ for some 
$\alpha_1>0$. For the 
statement of the theorem the values of $f$ are only important on the open 
interval $(0,\alpha_1)$
so that we can restrict $f$ to this interval for the assumptions.
\begin{theorem}\label{thm:bv}
Suppose $p_0 \in \mathcal P$ is bounded, has bounded support 
$[0,\alpha_1]$ and that $p_0(x)\ge\zeta>0$ for all $x\in[0,\alpha_1]$. Then for 
\[\mathcal B = \{f\in 
BV[0,\infty):f|_{(0,\alpha_1)}=g,\|g\|_\infty+\|Dg/Dp_0\|_{\infty,Dp_0} 
\le B\}\]
we have
\begin{align*}
 \sup_{f \in \mathcal B}|(\hat P_n -P_n)(f)| = O_{P_0^\mathbb N} (B
n^{-2/3})
\end{align*}
and consequently
\begin{align*}
 \sqrt n (\hat P_n-P_0) \to^d G_{P_0} \text{ in } \ell^\infty(\mathcal B).
\end{align*}
\end{theorem}
The proof of Theorem~\ref{thm:bv} is deferred to the end of 
the paper. 
Let us consider one particular example as a corollary.
We can take for points $t>0$ where $p_0$ is discontinuous 
the indicator function $f=\1_{[0,t]}$ in Theorem~\ref{thm:bv}. We have 
$Df=-\delta_t$.
Say $p_0$ has a discontinuity of size 
$\Delta\equiv
\lim_{s\nearrow t}p_0(s)-\lim_{s \searrow t} p_0(s)>0$ at $t$. Then we have
$Dp_0=-\Delta\delta_t-\mu$ for a positive measure~$\mu$. In this case we obtain
$\|Df/Dp_0\|_{\infty,Dp_0}=1/\Delta$ leading to 
the following corollary.
\begin{corollary}\label{cor:jump}
Suppose $p_0 \in \mathcal P$ is bounded, has bounded support 
$S_0$ and that $p_0(x) \ge\zeta>0$ for all $x\in S_0$. 
Then for each $t>0$ where $p_0$ is discontinuous we have
\begin{align*}
 |\hat F_n(t) -F_n(t)| 
= O_{P_0^\mathbb N} (n^{-2/3})
\end{align*}
and consequently
\begin{align*}
 \sqrt{n}(\hat F_n(t)-F_0(t))\to^d N(0,F_0(t)-F_0(t)^2).
\end{align*}
\end{corollary}
Here for $t>0$ such that $F_0(t)=1$ we understand $N(0,0)$ to be 
$\delta_0$.
To formulate the next theorem we define for $s>0$ the H\"older spaces
\begin{align*}
 &C^s([0,\infty))\\
&\equiv\bigg\{f\in 
C([0,\infty)):\|f\|_{C^s}=\sum_{j=0}^{[s]}\|D^j 
f\|_\infty+\sup_{x\neq y}\frac{|D^{[s]}f(x)-D^{[s]}f(y)|}{|x-y|^{s-[s]}}
<\infty\bigg\},
\end{align*}
where $[s]$ denotes the integer part of $s$ and $C([0,\infty))$ are the bounded 
real-valued continuous functions on $[0,\infty)$. We will see that under a 
strict curvature condition the set $\mathcal B$ in Theorem~\ref{thm:bv} 
contains $C^1$-H\"older balls. By strict curvature condition we have in mind 
that $p_0'$ is bounded away from zero, that is 
$\inf_{x\in S_0}|p_0'(x)|\ge\xi>0$ or equivalently 
$\sup_{x\in S_0}1/|p_0'(x)|\le1/\xi$. 
We do not want to assume that $p_0'$ exists classically.
To allow for discontinuities 
of $p_0$ and to stay in the general setting  of weak derivatives we assume that 
the Lebesgue measure on 
$S_0$ denoted by $\lambda$ is absolutely continuous with respect to $Dp_0$ and 
replace the 
assumption $\sup_{x\in S_0}1/|p_0'(x)|\le1/\xi$ by the weaker assumption
$\|\lambda/Dp_0\|_\infty\le1/\xi$, where $\lambda/Dp_0$ is the 
Radon--Nikodym derivative of $\lambda$ with respect to $Dp_0$.
We remark that $\|\lambda/Dp_0\|_{\infty,Dp_0}=\|\lambda/Dp_0\|_{\infty}$.
Let $\mathcal F$ be a $C^1$-H\"older ball and 
$g=f|_{(0,\alpha_1)}$ with $f\in\mathcal F$.
Then 
$\|Dg/Dp_0\|_{\infty,Dp_0}=\|Dg/\lambda \cdot \lambda/Dp_0\|_{\infty,Dp_0}\le 
(1/\xi) 
\|g'\|_\infty$ and we see that the $C^1$-H\"older ball $\mathcal F$ is 
contained in $\mathcal 
B$ for some $B$. This special case of Theorem~\ref{thm:bv} with $\mathcal 
F$ instead of $\mathcal B$
can be generalized to balls in H\"older spaces $C^s([0,\infty))$ of order 
$s>1/2$.
\begin{theorem} \label{LIKkw2}
Suppose $p_0 \in \mathcal P$ is bounded, has bounded support 
$S_0$ and that $p_0(x)\ge\zeta>0$ for all $x\in S_0$. 
Denote by $\lambda$ the Lebesgue measure on $S_0$ and let~$\lambda$ be 
absolutely continuous
with respect to $Dp_0$ and $\|\lambda/Dp_0\|_\infty<\infty$.
Let $\mathcal F$ be a ball in the $s$-H\"older space $C^s([0,\infty))$ of order 
$s>1/2$. 
Then
$$\sup_\mathcal F|(\hat P_n
-P_n)(f)| = o_{P_0^\mathbb
N}(1/\sqrt n) $$ as $n \to \infty$ and thus
$$\sqrt n (\hat P_n-P_0) \to^d G_{P_0} \text{ in } 
\ell^\infty(\mathcal F).$$
In particular, for
any $f \in C^s([0,\infty))$ with $s>1/2$ we have $$\sqrt n \int_0^\infty (\hat 
p_n(x)-p_0(x))f(x)dx \to^d N(0,
\|f-P_0f\|_{L^2(P_0)}^2).$$
\end{theorem}
The proof of Theorem~\ref{LIKkw2} will be given at the end of the paper.

\section{Discussion and applications}\label{sec:discussion}

The $n^{2/3}$-rate appearing in Theorem~\ref{thm:bv} and 
Corollary~\ref{cor:jump} is the pointwise rate at which the least concave 
majorant $\hat F_n$ converges to the empirical distribution function $F_n$. 
\cite{Wang1994} derived the pointwise limit theorem with a $n^{2/3}$-rate for 
$\hat F_n(t_0)$ at a point $t_0>0$ where $p_0$ has a negative derivative. 
Although our statement is of uniform nature we obtain the same rate in 
Theorem~1 and no additional logarithmic factor needs to be paid for the 
uniformity. 
A possible explanation is that the class of functions is adapted 
to the density $p_0$. We also note that Theorem~1 yields uniformity only over 
finitely many indicator functions even if $p_0$ has infinitely many 
discontinuities. The rate in the Kiefer--Wolfowitz theorem is $(n/\log 
n)^{2/3}$, which differs from our rate by a logarithmic factor. Indeed for 
bounding $\|\hat F_n-F_n\|_\infty$ this additional factor seems necessary at 
least it was shown by \cite{DurotTocquet2003} that it is necessary 
in the monotone regression 
framework.
Our results further differ from the Kiefer--Wolfowitz theorem in the sense that
the convergence is in probability, whereas the 
Kiefer--Wolfowitz theorem 
yields almost sure convergence. Our assumptions are 
weaker than in the Kiefer--Wolfowitz theorem since $p_0$ is neither assumed to 
be differentiable nor continuous. We require a strict curvature condition only 
in Theorem~\ref{LIKkw2}, but not in Theorem~\ref{thm:bv} or 
Corollary~\ref{cor:jump}.

Linear functionals of the Grenander estimator have also been studied 
by \cite{Jankowski2014} so let us discuss the differences in scope and in the 
assumptions between the results. 
A distinct feature of Theorems~\ref{thm:bv} and~\ref{LIKkw2} is that they are 
not for a fixed function $f$ but they are uniform in $f$ over classes of 
functions. 
\cite{Jankowski2014} takes a different perspective and emphasises the problem 
of possible misspecification, meaning that the true density $p_0$ 
does 
not necessarily need to be non-increasing. She distinguishes 
between curved and flat parts of $p_0$ (or in case of misspecification of its 
Kullback--Leibler projection) and assumes on the portion of support where $p_0$ 
is curved that $p_0$ is continuously differentiable and that $|p'_0|$ is 
bounded, which is used for the application of the Kiefer--Wolfowitz theorem in 
the proof.
The assumption that $p_0$ is continuously differentiable is widely used in the 
literature on the Grenander estimator so it is worthwhile to remark that 
Theorems~\ref{thm:bv} and~\ref{LIKkw2} do not require $p_0$ to be 
differentiable nor to be continuous.
The function $f$ defining the functional is assumed by \cite{Jankowski2014} to 
be differentiable on the curved parts, whereas Theorem~\ref{thm:bv} allows for 
discontinuities of $f$ at points where $p_0$ is discontinuous and 
Theorem~\ref{LIKkw2} only requires H\"older 
smoothness of order $s>1/2$ on the curved parts. \cite{Jankowski2014} assumes 
the function $f$ to be in $L^p$, $p>2$, on the flat 
parts, while Theorem~\ref{thm:bv} 
assumes $f$ to be constant
on the flat parts which in view of the results by 
\cite{Jankowski2014} is the natural condition to ensure a Gaussian limit. 
Theorem~\ref{LIKkw2} excludes flat parts by a strict curvature condition.
To summarise the comparison with Jankowski's results we can say that
our approach has the advantage of providing uniform 
results under low regularity assumptions on $p_0$ and requires 
stronger assumptions on $f$ on the flat parts while needing weaker assumptions 
on the curved parts.

As an application we present the estimation of sums of independent random 
variables~$X$ and~$Y$ with densities by $p_0$ and $q_0$, respectively. 
Let $X+Y=Z$ and the aim is the estimation of the density 
of $Z$. 
We observe either independent i.i.d. samples $X_1,\dots,X_n$ and 
$Y_1,\dots,Y_n$ of $X$ and $Y$, respectively, or in the special case, where $X$ 
and $Y$ have the same distribution, we only need one sample and can set 
$X_j=Y_j$ for $j=1,\dots,n$.
The random variable $Z$ has the density $p_0\ast q_0$ and the canonical 
estimator is $\hat p_n\ast \hat q_n$, where~$\hat p_n$ and~$\hat q_n$ are 
estimators of $p_0$ and $q_0$, respectively. 
For kernel estimators the convergence of $\hat p_n\ast \hat q_n-p_0\ast q_0$ in 
distribution in $L^1(\mathbb R)$ was shown by 
\cite{SchickWefelmeyer2004,SchickWefelmeyer2007} as well as
\cite{GineMason2007}.
\cite{Nickl2009} derives general conditions for the convergence of convolutions 
of density estimators which we verify by our uniform central limit theorems for 
the Grenander estimator. To this end let $p_0:\mathbb R\to[0,\infty)$ and 
$q_0:\mathbb R\to[0,\infty)$ be densities of random variables such that 
$p_0|_{[0,\infty)}$ and $q_0|_{[0,\infty)}$ satisfy the assumptions of 
Theorem~\ref{LIKkw2}. 
We denote by $\hat p_n$ and $\hat q_n$ the respective Grenander estimators.
We decompose
\begin{equation}
  \begin{aligned}\label{eq:decomp}
 \sqrt{n}(\hat p_n \ast \hat q_n - p_0\ast q_0)
&=\sqrt{n}(\hat p_n-p_0)\ast q_0+\sqrt{n}(\hat q_n-q_0)\ast p_0\\
&\quad+\sqrt{n}(\hat 
p_n-p_0)\ast(\hat q_n-q_0).
  \end{aligned}
\end{equation}
By Theorem~\ref{LIKkw2} we have uniform central limit theorems for $\hat P_n$ 
and $\hat Q_n$ in $\ell^\infty(\mathcal F)$ for balls $\mathcal F$ in the 
H\"older space $C^s$, $s>1/2$.
By Lemma~8(b) in \cite{GineNickl2008} we infer 
from $p_0$, $q_0$ being of bounded variation and in $L^1(\mathbb R)$ that 
$p_0,q_0\in B^1_{1\infty}(\mathbb R)\hookrightarrow B^s_{11}(\mathbb R)$ for 
$s<1$, where $B^s_{pq}(\mathbb R)$ are Besov spaces.
Together these two statements yield the convergence of $\sqrt{n}(\hat 
p_n-p_0)\ast q_0$ in distribution in~$L^1(\mathbb R)$ by Theorem~2 in 
\cite{Nickl2009} 
and likewise for the term $\sqrt{n}(\hat q_n-q_0)\ast p_0$.
To bound the last term in \eqref{eq:decomp} we use Young's inequality, the 
bounded support of $p_0$ and $q_0$ as well as Proposition~\ref{propRate} below
\begin{align*}
\|(\hat p_n-p_0)\ast(\hat q_n-q_0)\|_1
&\le \|\hat p_n-p_0\|_1 \|\hat q_n-q_0\|_1
\le C \|\hat p_n-p_0\|_2 \|\hat q_n-q_0\|_2\\
&=O_{P_0^\mathbb N}(n^{-2/3})
=o_{P_0^\mathbb N}(n^{-1/2}),
\end{align*}
where $C>0$ is some constant.
We conclude that \[\sqrt{n}(\hat p_n\ast \hat q_n-p_0\ast q_0)\]
converges in distribution in $L^1(\mathbb R)$.
By Remark~2 in \cite{Nickl2009} the limiting random variable can be determined 
by calculating for every $h\in L^\infty=(L^1)^\ast$ the limits of
\[\sqrt{n}\int h(x) d\Big((\hat P_n-P_0)\ast Q_0+(\hat Q_n-Q_0)\ast 
P_0\Big)(x).\]
Moreover, it follows from our Theorem~2 and from the continuous linear mapping 
established in the proof of Theorem~2 in \cite{Nickl2009} that the limiting 
random variable may likewise be determined by calculating for all $h\in 
L^\infty=(L^1)^\ast$ the limits of 
\begin{align*}
  &\sqrt{n}\int h(x) d\Big((P_n-P_0)\ast Q_0+(Q_n-Q_0)\ast P_0\Big)(x)\\
&=\sqrt{n}\left(\frac1{n}\sum_{j=1}^n 
g(X_j)-E_{P_0}g\right)+\sqrt{n}\left(\frac1{n}\sum_{j=1}^n 
f(Y_j)-E_{Q_0}f\right)
\end{align*} 
with $g=q_0(-\cdot)\ast h$ and $f=p_0(-\cdot)\ast h$. So the limit 
is a mean zero Gaussian random variable with an explicitly given covariance 
structure.

The uniform results can be interpreted in the context of
\cite{BickelRitov2003}. They coin the expression ``plug-in
estimator'' for a rate optimal estimator of
a density which simultaneously leads to the efficient estimation of functionals
with uniform convergence over a class of functionals. The Grenander estimator 
attains the optimal $n^{1/3}$-rate for non-increasing 
densities.
Since~$\mathcal P$ is a nonparametric model,
the empirical distribution 
function is an asymptotically efficient estimator of $P_0$ considered as an 
element $f\mapsto P_0 f$ of the space $\ell^\infty(\mathcal F)$ for a Donsker 
class $\mathcal F$, see \citet[p.~420]{vanderVaartWellner1996}. By our results 
$\hat P_n$ and $P_n$ are closer than $P_n$ and 
$P_0$ so that $\hat P_n$ is an asymptotically efficient estimator as well.
In summary our results show that the Grenander estimator 
is a plug-in estimator for a subclass of the bounded variation functions and 
for H\"older balls of smoothness $s>1/2$.

\section{The derivative of the likelihood function}\label{sec:deriv}

Many classical properties of maximum likelihood estimators $\hat \theta_n$ of
regular parameters $\theta \in \Theta \subset \mathbb R^p$, such as asymptotic
normality, are derived from the fact that the derivative of the log-likelihood
function vanishes at $\hat \theta_n$, 
\begin{equation} \label{LIKpmle}
\frac{\partial}{\partial \theta} \ell_n(\theta) _{|\hat \theta_n} =0.
\end{equation}
 This typically relies on the assumption that the true parameter $\theta_0$ is
interior to $\Theta$ so that by consistency $\hat \theta_n$ will then eventually
also be. In the infinite-dimensional setting, even if one can define an
appropriate notion of derivative, this approach is usually not viable since
$\hat p_n$ is \textit{never} an interior point in the parameter space even when
$p_0$ is. 

We now investigate these matters in more detail in the setting where $\mathcal
P$ consists of bounded probability densities. In this case we can compute the
Fr\'echet derivatives of the log-likelihood function on the space
$L^\infty=L^\infty([0,\infty))$ equipped with the $\|\cdot\|_\infty$-norm.
Recall that a real-valued function $L: U \to \mathbb R$ defined on an open
subset $U$ of a Banach space $B$ is Fr\'echet differentiable at  $f \in U$ if
\begin{equation}
\lim_{\|h\|_B \to 0} \frac{|L(f+h)-L(f)-DL(f)[h]|}{\|h\|_B} =0
\end{equation}
for some linear continuous map $DL(f): B \to \mathbb R$. If $g \in U$ is 
such
that the line segment $(1-t)f+tg, t \in (0,1),$ joining $f$ and $g$ lies in
$U$ (for instance if $U$ is convex) then the directional derivative of $L$ at
$f$ in the direction $g$ equals precisely  $$ \lim_{t \to 0+} \frac{L(f+
t(g-f))-L(f)}{t} = DL(f)[g-f].$$ 
The second order Fr\'echet derivatives are defined by taking the Fr\'echet 
derivative of $DL(f)[h]$ for a fixed direction $h$ and likewise higher order 
Fr\'echet derivatives are defined.
The following proposition shows that the log-likelihood function $\ell_n$ is
Fr\'echet differentiable on the open convex subset of $L^\infty$ consisting of
functions that are positive at the sample points. A similar result holds for
$\ell$ if one restricts to functions that are bounded away from zero on 
the support~$S_0$ of~$p_0$.
We recall here these results of Proposition~3 in \cite{Nickl2007}.

\begin{proposition} \label{LIKcalc}
For any finite set of points $x_1, \dots, x_n \in [0,\infty)$ define $$\mathcal
U(x_1, \dots, x_n) = \left\{f \in L^\infty([0,\infty)): \min_{1 \le i \le n} 
f(x_i)>0
\right\}$$ and $$\mathcal U = \left\{f \in L^\infty([0,\infty)): \inf_{x \in 
S_0}f(x)>0\right\}.$$ Then $\mathcal U(x_1, \dots, x_n)$ and $\mathcal U$ 
are 
open
subsets of $L^\infty([0,\infty))$.

Let $\ell_n$ be the log-likelihood function  from (\ref{LIKlog}) based on $X_1,
\dots, X_n \sim^{i.i.d.}P_0$, and denote by $P_n$ the empirical measure
associated with the sample. Let $\ell$ be as in (\ref{LIKllog}). For $\alpha \in
\mathbb N$ and $f_1, \dots, f_\alpha \in L^\infty([0,\infty))$ the $\alpha$-th 
Fr\'echet
derivatives of $\ell_n : \mathcal U(X_1, \dots, X_n) \to \mathbb R, ~\ell:
\mathcal U \to \mathbb R$ at a point $f \in \mathcal U(X_1, \dots, X_n),~f \in
\mathcal U$, respectively, are given by
\begin{align}
 D^\alpha\ell_n(f)[f_1, \dots, f_\alpha] &\equiv (-1)^{\alpha-1} (\alpha-1)!
P_n(f^{-\alpha}f_1 \cdots f_\alpha),\label{LIKcalcEmp}\\
D^\alpha\ell(f)[f_1, \dots, f_\alpha] &\equiv (-1)^{\alpha-1} (\alpha-1)!
P_0(f^{-\alpha}f_1 \cdots f_\alpha).\label{LIKcalcEqu} 
\end{align}
\end{proposition}

\medskip

We deduce from the above proposition the intuitive fact that the limiting
log-likelihood function has a derivative at the true point $p_0>0$ that is zero
in all `tangent space' directions $h$ in
\begin{equation}
\mathcal H \equiv \left\{h: \int_{S_0} h=0\right\}
\end{equation}
since
\begin{equation} \label{LIKtan}
D\ell(p_0)[h] = \int_{S_0} p_0^{-1} h dP_0 = \int_{S_0} h =0.
\end{equation}
However, in the infinite-dimensional setting the empirical counterpart of
(\ref{LIKtan}),
\begin{equation} \label{LIKscoreh}
D\ell_n(\hat p_n)[h] =0
\end{equation}
for $h \in \mathcal H$ and $\hat p_n$ the nonparametric maximum likelihood
estimator is not true in general. Even if the set $\mathcal P$ the likelihood
was maximised over is contained in $\mathcal U(X_1, \dots, X_n)$ it will itself
in typical nonparametric situations have empty interior in $L^\infty$, and the
maximiser $\hat p_n$ will lie at the boundary
of $\mathcal P$. As a consequence we cannot expect that $\hat p_n$ is a zero of
$D\ell_n$.  

Following ideas in \cite{Nickl2007} we can circumvent this problem in some 
situations:
if the true value $p_0$
lies in the `interior' of $\mathcal P$ in the sense that local
$L^\infty$-perturbations of $p_0$ are contained in $\mathcal P \cap \mathcal U
(X_1, \dots, X_n)$, then we can bound $D\ell_n$ at $\hat p_n$.

\begin{lemma} \label{LIKtrick}
Let $\hat p_n$ be as in (\ref{LIK0}) and suppose that for some $h \in 
L^\infty([0,\infty)),
\eta>0,$ the line segment joining $\hat p_n$ and $p_0 \pm \eta h$ is contained
in $\mathcal P \cap \mathcal U(X_1, \dots, X_n)$. Then
\begin{equation} \label{LIKtrs}
|D\ell_n(\hat p_n)[h]| \leq (1/\eta) |D\ell_n(\hat p_n)[\hat p_n-p_0]|.
\end{equation}
\end{lemma}
\begin{proof}
Since $\hat p_n$ is a maximiser over $\mathcal P$ we deduce from
differentiability of $\ell_n$ on $\mathcal U(X_1, \dots, X_n)$ that the
derivative at $\hat p_n$ in the direction $p_0 + \eta h \in \mathcal P  \cap
\mathcal U(X_1, \dots, X_n)$ necessarily has to be nonpositive, that is
\begin{equation}
D\ell_n(\hat p_n)[p_0 + \eta h-\hat p_n] = \lim_{t \to 0+} \frac{\ell_n(\hat p_n
+ t(p_0+\eta h - \hat p_n)) - \ell_n(\hat p_n)}{t} \le 0
\end{equation}
or, by linearity of $D\ell_n(\hat p_n)[\cdot]$,
\begin{equation}
D\ell_n(\hat p_n)[ \eta h] \le D\ell_n(\hat p_n)[\hat p_n-p_0].
\end{equation}
Applying the same reasoning with $-\eta$ we see
\begin{equation}
|D\ell_n(\hat p_n)[ \eta h]| \le D\ell_n(\hat p_n)[\hat p_n-p_0] = |D
\ell_n(\hat p_n)[\hat p_n-p_0]|.
\end{equation}
Divide by $\eta$  to obtain the result.
\end{proof}

\medskip

The above lemma is interesting if we are able to show that $$D\ell_n(\hat
p_n)[\hat p_n-p_0]=o_{P_0^\mathbb
N}(1/\sqrt n),$$ as then the same rate bound carries over to
$D\ell_n(\hat p_n)[h]$. This can in turn be used to mimic the finite-dimensional
asymptotic normality proof of maximum likelihood estimators, which does not
require (\ref{LIKpmle}) but only that the score is of smaller stochastic order
of magnitude than $1/\sqrt n$. As a consequence we will be able to obtain the
asymptotic distribution of linear integral functionals of~$\hat p_n$, and more
generally, for $\hat P_n$ the probability measure associated with~$\hat p_n$,
central limit theorems for $\sqrt n (\hat P_n-P)$ in `empirical process - type'
spaces $\ell^\infty (\mathcal F)$. To understand this better we notice that
Proposition \ref{LIKcalc} implies the following relationships: If we define the
following projection of $f \in L^\infty$ onto $\mathcal H$,
\begin{equation} \label{LIKproj}
\pi_0(f) \equiv (f-P_0f)p_0 \in \mathcal H,~~P_0(f) = \int_{0}^{\infty} f dP_0,
\end{equation}
and if we assume $p_0>0$ on $S_0$ then
$$\int_0^\infty (\hat p_n-p_0)f dx = \int_{S_0}  p_0^{-2} (\hat p_n-p_0)
(f-P_0f)p_0 dP_0 =  -D^2 \ell(p_0)[\hat p_n-p_0, \pi_0(f)]$$ and
$$D\ell_n(p_0)[\pi_0(f)] = (P_n-P_0)f$$ so that:

\begin{lemma} \label{LIKclas} Suppose $p_0>0$ on $S_0$. Let $\hat p_n$ be as in
(\ref{LIK0}) and let $\hat P_n$ be the random probability measure induced by
$\hat p_n$. For any $f \in L^\infty([0,\infty))$ and $P_n$ the empirical 
measure we have
\begin{equation}
  \begin{aligned}
|(\hat P_n-P_n)(f)| &=\bigg|\int_0^\infty f d(\hat P_n - P_n)\bigg| \\
&=
|D\ell_n(p_0)[\pi_0(f)] + D^2 \ell(p_0)[\hat p_n-p_0, \pi_0(f)]|.
\end{aligned}
\end{equation}
\end{lemma}
Heuristically the right hand side equals, up to second order
\begin{equation} 
D\ell_n(\hat p_n)[\pi_0(f)] - D^2\ell_n(p_0)[\hat p_n - p_0, \pi_0(f)] + D^2
\ell(p_0)[\hat p_n-p_0, \pi_0(f)].
\end{equation}
Control of (\ref{LIKtrs}) at a rate $o_{P_0^\mathbb
N}(1/\sqrt n)$ combined with stochastic
bounds on the second centred log-likelihood derivatives and convergence rates
for $\hat p_n -p_0 \to 0$ thus give some hope that one may be able to prove
$$(\hat P_n-P_0 - P_n+P_0)(f) = (\hat P_n-P_n)(f) = o_{P_0^\mathbb
N}(1/\sqrt n)$$ and that
thus, by the central limit theorem for $(P_n-P_0)f$, $$\sqrt n \int_0^\infty
(\hat p_n-p_0)f dx \to^d N(0, P_0(f-P_0f)^2)$$ as $n \to \infty$.

\section{Bounding the estimator and \texorpdfstring{$L^2$}{L\texttwosuperior\ 
}-convergence rate}\label{sec:bounds}

We establish some first probabilistic properties of $\hat p_n$ that will be
useful below: If $p_0$ is bounded away from zero on $S_0$ then so is $\hat p_n$ 
on the
interval $[0,X_{(n)}]$, where $X_{(n)}$ is the last order statistic. Similarly
if $p_0$ is bounded above then so is $\hat p_n$ with high probability.

\begin{lemma} \label{LIKlow}
a) Suppose the true density $p_0$ has compact support $S_0$ and that 
$\inf_{x \in 
S_0}p_0(x)\ge\zeta>0$. Then,
for every $\epsilon >0$, there exists $\xi >0$ and a finite index $N(\epsilon)$
such that, for all $n \geq N(\epsilon)$, 
\begin{equation*}
P_0^\mathbb N \left(\inf_{x \in [0,X_{(n)}]} \hat p_n(x) < \xi \right) = \Pr
\left(\hat p_n(X_{(n)}) < \xi \right) < \epsilon.
\end{equation*}
b) Suppose the true density $p_0$ satisfies $p_0(x) \leq K < \infty$ for all $x
\ge0$. Then, for every $\epsilon >0$, there exists $0<k< \infty$ such
that for all $n\in\mathbb N$ 
\begin{equation*}
P_0^\mathbb N \left(\sup_{x \ge0} \hat p_n(x) > k \right) = \Pr \left(\hat
p_n(0) > k \right) < \epsilon.
\end{equation*}
\end{lemma}
\begin{proof}
a) The first equality is obvious, since $\hat p_n$ is monotone decreasing. 
Let $X_{(1)},\dots,X_{(n)}$ denote the order statistic of $X_1,\dots,X_n$.
On
each of the intervals $(X_{(j-1)}, X_{(j)}]$, $f_n$ is the slope of the least 
concave
majorant of $F_n$. The least concave majorant connects $(X_{(n)},1)$
and at least one other order statistic (possibly $(X_{(0)},0)\equiv(0,0)$), so 
that
$$\{\hat
p_n(X_{(n)}) < \xi\} \subseteq \left\{X_{(n)}-X_{(n-j)} > j/(\xi n)
~~\textrm{for some}~~ j=1,\dots,n \right \}.$$
Note next that  since $F_0$ is strictly monotone on $S_0$ we have $X_{i} =
F_0|_{S_0}^{-1}F_0(X_i)$ and 
\begin{align*}
  F_0|_{S_0}^{-1}F_0(X_{(n)}) - 
F_0|_{S_0}^{-1}F_0(X_{(n-j)}) 
&\leq 
(\inf_{x \in 
S_0}p_0(x))^{-1} \left(F_0(X_{(n)}) - F_0(X_{(n-j)})\right) \\
&\leq \zeta^{-1}
\left(U_{(n)}-U_{(n-j)}\right),
\end{align*}
where the $U_{(i)}$'s are distributed as the
order statistics of a sample of size $n$ of a uniform random variable on
$[0,1]$, and where $U_{(0)}=0$ by convention. Hence it suffices to bound 
\begin{equation} \label{unif}
\Pr \left(U_{(n)}-U_{(n-j)} > \frac{\zeta j}{\xi n} ~~\textrm{for some}~~
j=1,\dots,n \right).
\end{equation}
By a standard computation involving order statistics, the joint
distribution of $U_{(i)}$, $i=1,\dots,n$, is the same as the one of 
$Z_i/Z_{n+1}$
where $Z_n = \sum_{l=1}^n W_l$ and where $W_l$ are independent standard
exponential random variables. Consequently, for $\delta >0$, the probability in
(\ref{unif}) is bounded by 
\begin{align*}
& \Pr \left( \frac{W_{n-j+1}+\dots+W_n}{Z_{n+1}} > \frac{\zeta j}{\xi n}
~~\textrm{for some}~~ j \right) \\
&= \Pr \left( \frac{n}{Z_{n+1}} \frac{W_{n-j+1}+\dots+W_n}{n} > \frac{\zeta
j}{\xi n} ~~\textrm{for some}~~ j \right) \\
&\le \Pr (n/Z_{n+1} > 1 + \delta) + \Pr \left(\frac{W_{n-j+1}+\dots+W_n}{n} >
\frac{\zeta j}{\xi n (1+\delta)} ~~\textrm{for some}~~ j  \right) \\
& = A +B.
\end{align*}
To bound A, note that it is equal to $$\Pr \left(\frac{1}{n+1} \sum_{l=1}^{n+1}
(W_l-EW_l) < \frac{- \delta -(1+\delta)/n}{1+\delta} \frac{n}{n+1} \right),$$
which, since $\delta>0$, is less than $\epsilon/2>0$ arbitrary, from some $n$
onwards, by the law of large numbers. For the term B we have, for $\xi$ small
enough and by Markov's inequality
\begin{eqnarray*}
&& \Pr \left(W_{n-j+1}+\dots+W_n > \frac{\zeta j}{\xi (1+\delta)} ~~\textrm{for
some}~~ j  \right) \\
&& \le \sum_{j=1}^n \Pr \left(W_{n-j+1}+\dots+W_n > \frac{\zeta j}{\xi 
(1+\delta)}
\right) \\
&&= \sum_{j=1}^n \Pr \left(\sum_{l=1}^j ( W_{n-l+1} - EW_{n-l+1}) > \frac{\zeta
j}{\xi (1+\delta)} -j \right) \\
&& \le \sum_{j=1}^n \frac{\xi^4 E (\sum_{l=1}^j (W_{n-l+1} - EW_{n-l+1}))^4}{j^4
C(\delta, \zeta)} \\
&& \le \xi^4 C'(\delta, \zeta) \sum_{j=1}^n j^{-2} \le \xi^4 C''(\delta, \zeta)
<
\epsilon/2,
\end{eqnarray*}
since, for $Y_l=W_{n-l+1} - EW_{n-l+1}$, by Hoffmann-J\o rgensen's inequality
\citep[Corollary~1.2.7]{delaPenaGine1999} $$\left \|
\sum_{l=1}^j Y_l \right \|_{4,P} \le K \left[ \left \|\sum_{l=1}^j Y_l \right
\|_{2,P} + \left \| \max_l Y_l \right \|_{4,P}  \right] \le K' \left(\sqrt {j} +
j^{1/4} \right), $$ using the fact that $EW_1^p = p!$ and $Var (Y_1) = 1$. 

b) Since $\hat p_n$ is the left derivative of the least
concave majorant of the empirical distribution~$F_n$, $$\|\hat
p_n\|_\infty = \hat p_n(0) >M ~\iff~F_n(t)>Mt~~\text{for some } t.$$ Since $F_0$
is concave and continuous it maps $[0,\infty)$ onto $[0,1]$ and satisfies 
$F_0(t) \le
p_0(0)t \le t\|p_0\|_\infty$ so that we obtain
\begin{equation}
  \begin{aligned}\label{uniformDist}
 P^\mathbb N_0(\|\hat p_n\|_\infty >M) &\le P_0^\mathbb N
\left(\sup_{t>0}\frac{F_n(t)}{F_0(t)} >M/\|p_0\|_\infty \right) \\
&=  
P_0^\mathbb
N
\left(\sup_{t>0}\frac{F^U_n(t)}{t} >M/\|p_0\|_\infty \right),
  \end{aligned}
\end{equation}
where $F_n^U$ is the empirical distribution function based on a sample of size
$n$ from the uniform distribution.
The density of the order statistic of $n$ uniform $U(0,1)$ random variables is
$n!$ on the set of all $0\le x_1<\dots<x_n\le 1$.
Let
$M\ge\|p_0\|_\infty$ and set $C\equiv M/\|p_0\|_\infty$ then the complement 
of the event in \eqref{uniformDist}
has
the
probability
\begin{align*}
& P_0^\mathbb N
\left(\frac{F^U_n(t)}{t} \le C \quad\forall t\in[0,1]\right)\\
&=n!\int_{1/C}^{1}\int_{(n-1)/(nC)}^{x_n}\dots\int_{1/(nC)}^{x_2}dx_1\dots
dx_{n-1 }dx_n\\
&=n!\int_{1/C}^{1}\dots
\int_{j/(nC)}^{x_{j+1}}\frac{1}{(j-1)!}x_j^{j-1}-\frac{1}{nC}\frac{1}{(j-2)!}
x_j^{(j-2)}dx_j
\dots
dx_n\\
&=1-\frac{1}{C},
\end{align*}
where $j=2,\dots,n-1$. In particular, the probability in \eqref{uniformDist}
equals $\|p_0\|_\infty/M$ and can be made small by choosing $M$ large. 
\end{proof}

\medskip

We can now derive the rate of convergence of the maximum likelihood estimator of
a monotone density. The rate corresponds to functions that are once
differentiable in an $L^1$-sense, which is intuitively correct since a monotone
decreasing function has a weak derivative that is a finite signed measure.
The following convergence rate in Hellinger distance is given
in Example~7.4.2 by \cite{vandeGeer2000}. It is used here to derive the
$L^2$-convergence rate.
\cite{KulikovLopuhaae2005} prove a much finer result by deriving the asymptotic
distribution of the $L^p$-errors under stronger smoothness assumptions. 
\cite{GaoWellner2009} consider the maximum likelihood estimator of a 
$k$-monotone density on a bounded interval and extend Lemma \ref{LIKlow}b) 
and the Hellinger convergence rate in the following proposition to this setting.
We recall that the Hellinger distance between two Lebesgue densities $p$ and 
$q$ 
is defined by
\[
 h^2(p,q)=\frac1{2}\int\left(p^{1/2}(x)-q^{1/2}(x)\right)^2 
dx.
\]

\begin{proposition}\label{propRate}
Suppose $p_0 \in \mathcal P^{mon}$ and that $p_0$ is bounded and has bounded 
support. Let $\hat p_n$
satisfy (\ref{LIK0}). Then 
\begin{align}
 h(\hat p_n, p_0)=O_{P_0^\mathbb N}(n^{-1/3})
\end{align}
and also 
\begin{align}\label{LIKgrenrat}
 \|\hat p_n- p_0\|_2=O_{P_0^\mathbb N}(n^{-1/3}).
\end{align}
\end{proposition}
\begin{proof}
Since $p_0$ is bounded and has bounded support the statement for the Hellinger 
distance is contained in Example~7.4.2 by~\cite{vandeGeer2000}.
The density~$p_0$ is bounded by assumption and we have
$\|\hat p_n\|_\infty=O_{P_0^\mathbb N}(1)$ by Lemma~\ref{LIKlow}b).
Then the result in
$L^2$-distance follows by the bound
\begin{align*}
 \|\hat 
p_n-p_0\|_2^2&\le\int\left(\hat p_n^{1/2}(x)-p_0^{1/2}(x)\right)^2\left(\hat 
p_n^{1/2} (x)+p_0^ { 1/2 } (x)\right)^2 dx\\
&\le2(\|\hat p_n\|_\infty^{1/2}+\|p_0\|_\infty^{1/2})^2h^2(\hat p_n,p_0)\\
&\le4(\|\hat p_n\|_\infty+\|p_0\|_\infty)h^2(\hat p_n,p_0).
\end{align*}
\end{proof}

\section{Putting things together}\label{sec:proofs}

The maximiser
$\hat p_n$ is in some sense an object that lives on the boundary of $\mathcal P$
-- it is piecewise constant with step-discontinuities at the observation points,
exhausting the possible `roughness' of a monotone function. 

We can construct line segments in the parameter space through $p_0$, following
the philosophy of Lemma \ref{LIKtrick}. 
Let $p_0$ be a non-increasing density with compact support $S_0$,  
$\inf_{x\in S_0}p_0(x) \ge \zeta >0$ and weak derivative $D p_0$.
In order to ensure that the perturbed function lies again in $\mathcal P$ we
will perturb $p_0$ by $\eta h$ where $h\in
L^\infty$, $\text{supp}(h)\subseteq S_0$, $\int h=0$ and $Dh$ is absolutely
continuous with respect to $Dp_0$ such that the Radon--Nikodym density
satisfies $\|Dh/Dp_0\|_{\infty,Dp_0}<\infty$. Then we have indeed for $\eta$ of
absolute value small enough
\begin{equation} \label{LIKint2}
\inf_{x\in S_0}(p_0 + \eta h)(x) \ge \zeta - \eta \|h\|_\infty > 0, 
~~~~\int_0^1 (p_0 + \eta h) =1,
\end{equation}
and that $D(p_0+\eta h)= Dp_0 + \eta Dh$ is a negative measure. We change 
$p_0+\eta h$ on a nullset so that it is equal to the integral of $Dp_0 + \eta 
Dh$ everywhere and thus is a non-increasing function.
Similar statements hold if we replace
$h$ by $\pi_0(f)$ defined in \eqref{LIKproj} when
$\|f\|_\infty+\|Df/Dp_0\|_{\infty,Dp_0}$ is finite. We possibly modify
$p_0+\eta\pi_0(f)$ on a nullset so that it equals the integral of its weak 
derivative.
\begin{lemma}\label{LIKpert}
Let $p_0$ be non-increasing and have bounded support $S_0$ with $K\ge 
p_0(x)\ge\zeta>0$ for all $x\in S_0$.
Let $f$ be such
that $\|f\|_\infty+\|Df/Dp_0\|_{\infty,Dp_0}$ is finite. Then we have
$p_0+\eta\pi_0(f)\in \mathcal P\cap \mathcal U$ for $|\eta|\le c
(\|f\|_\infty+\|Df/Dp_0\|_{\infty,Dp_0})^{-1}$, where $c>0$ depends on $K$ and
$\zeta$ only.
\end{lemma}
\begin{proof}
$\pi_0(f)=(f-P_0 f)p_0$ is bounded by $2K\|f\|_\infty$. The assumption
$p_0(x)\ge\zeta>0$ for all $x\in S_0$ yields $p_0+\eta\pi_0(f)\in \mathcal U$ 
for $|\eta|<
\zeta/(2K\|f\|_\infty)$. In addition to
$|\eta|<\zeta/(2K\|f\|_\infty)$ we will choose $\eta$ small enough such that
\begin{align*}
 D(p_0+\eta\pi_0(f))=(1-\eta P_0 f+\eta f)Dp_0+\eta p_0Df
\end{align*}
is a negative measure.
This is the case if 
\begin{align*}
 \frac{K|\eta| \|Df/Dp_0\|_{\infty,Dp_0}}{1-2|\eta| \|f\|_\infty}\le1
\quad
\Leftrightarrow \quad
(K
\|Df/Dp_0\|_{\infty,Dp_0}+2\|f\|_\infty)|\eta|\le1,
\end{align*}
which holds for $|\eta|\le
(\max(2,K)(\|f\|_\infty+\|Df/Dp_0\|_{\infty,Dp_0}))^{-1}$.
\end{proof}

For $p_0$ and $f$ as above we can apply Lemma~\ref{LIKtrick} with $h=\pi_0(f)$, 
where the line segment between $\hat p_n$ and $p_0\pm \eta \pi_0(f)$ being in 
$\mathcal P\cap\mathcal U(X_1,\dots,X_n)$ is guaranteed by Lemma~\ref{LIKlow}a) 
provided $p_0\pm \eta \pi_0(f)\in\mathcal P\cap\mathcal U(X_1,\dots,X_n)$.
We thus obtain that
on events of probability as close to
one as desired and for $n$ large enough,
\begin{equation} \label{LIKtrick2}
|D\ell_n(\hat p_n)[\pi_0(f)]| \le C (\|f\|_{\infty}+\|Df/Dp_0\|_{\infty,Dp_0})
|D\ell_n(\hat p_n)[\hat p_n -p_0]|
\end{equation}
for some constant $C$ that depends on $K$ and $\zeta$ only.

We next need to derive stochastic bounds of the likelihood derivative at $\hat
p_n$ in the direction of $p_0$.
\begin{lemma} \label{LIKgsc}
Suppose $p_0$ is bounded, has bounded support $[0,\alpha_1]$ and satisfies 
$\inf_{x \in 
[0,\alpha_1]} p_0(x) >0$. For~$\hat
p_n$ satisfying (\ref{LIK0}) we have $$|D\ell_n (\hat p_n)[\hat p_n -p_0]| =
O_{P_0^\mathbb N} (n^{-2/3}) $$
\end{lemma}
\begin{proof}
By Lemma~\ref{LIKlow} we can restrict to an event where 
\[0<\xi \le \inf_{x \in [0, X_{(n)}]}\hat p_n(x) \le\sup_{x\in[0,\infty)}\hat
p_n(x)\le k <\infty\]
 and by \eqref{LIKgrenrat} further to an event where $$\|\hat p_n
-p_0\|_{2,P_0} \le \|p_0\|_\infty^{1/2} \|\hat p_n - p_0\|_{2} 
\le \|p_0\|_\infty^{1/2}
M
n^{-1/3}$$ for some finite constant $M$.
For any $\delta_n\to0$ with $n\delta_n\to\infty$ and some $c>0$
\begin{align*}
\Pr((\alpha_1-X_{(n)})>\delta_n)&=\Pr((\alpha_1-\delta_n)>X_{(n)}
)\\
&=(F_0(\alpha_1-\delta_n))^n\le 
(1-c \delta_n)^n\to0,
\end{align*}
in particular we obtain for $\delta_n=\log n/n$ that
\begin{align}\label{OP}
 \alpha_1-X_{(n)}=O_{P^{\mathbb
N}_0}\left(\frac{\log n}{n}\right).
\end{align}
Let us define the random function
$\tilde p_n^{-1}\equiv \hat p_n^{-1}$ on $[0, X_{(n)}]$ and zero on
$(X_{(n)},\infty)$. By $D\ell_n(\tilde p_n)$ and $D\ell(\tilde p_n)$ we 
denote 
the corresponding right hand sides
in~\eqref{LIKcalcEmp} and~\eqref{LIKcalcEqu}.
We observe that $D\ell_n(\hat p_n)=D\ell_n(\tilde p_n)$.
The function
$h\equiv\tilde p_n^{-1}(\hat p_n-p_0)$ on $[0,X_{(n)}]$ and $h\equiv0$ 
elsewhere is 
of bounded variation
with norm $\|h\|_{BV}\equiv\|h\|_1+|Dh|(\mathbb R)$
bounded by a fixed constant $C$ that depends only on $k, \xi,
\|p_0\|_\infty$ and $\alpha_1$. We observe that $D\ell(p_0)[\hat p_n-p_0]=0$ 
by~\eqref{LIKcalcEqu} 
and obtain
\begin{align}\label{LIKsbd}
& |D\ell_n (\hat p_n)[\hat p_n -p_0]| \notag \\
& = |D\ell_n (\tilde p_n)[\hat p_n -p_0] - D\ell(\tilde p_n)[\hat p_n-p_0] +
(D\ell(\tilde p_n)-D\ell(p_0))[\hat p_n-p_0] | \notag \\
& \lesssim \sup_{h: \|h\|_{BV} \le C, \|h\|_{2,P_0} \leq \bar
M n^{-1/3}}
|(P_n-P_0)(h)|  + \|\hat p_n -p_0\|_2^2 + \int_{X_{(n)}}^{\alpha_1} |\hat p_n 
-p_0|
 \notag \\
& = O_{P_0^\mathbb N} \left(n^{-1/2} n^{-1/6} + n^{-2/3} + \frac{\log
n}{n}\right),
\end{align}
where we have used Theorem~3.1 in \cite{GineKoltchinskii2006} with $$H=id,
\sigma = \bar M  n^{-1/3}, F=const$$ combined with the bracketing
entropy bound for monotone functions
\citep[Theorem~2.7.5]{vanderVaartWellner1996} and its straight forward
generalisation to bounded variation functions to  control the supremum of the
empirical
process, \eqref{LIKgrenrat} to control the second term, and \eqref{OP} for the
last integral.
\end{proof}

\begin{proposition}\label{propBV}
Suppose $p_0 \in \mathcal P$ is bounded, has bounded support $S_0$ and 
satisfies $p_0(x) \ge\zeta>0$ for all
$x\in S_0$. Let $f \in L^\infty$ be such that $\|Df/Dp_0\|_{\infty,Dp_0}$ is 
finite. Then
$$|D \ell_n(\hat p_n)[\pi_0(f)]| =  O_{P_0^\mathbb N}
\left((\|f\|_\infty+\|Df/Dp_0\|_{\infty,Dp_0}) n^{-2/3} \right).$$
\end{proposition}
\begin{proof}
 By Lemma~\ref{LIKpert} we have that $p_0 + \eta \pi_0(f) \in \mathcal P \cap
\mathcal U$ for $\eta$ a small multiple of 
$\|f\|_\infty+\|Df/Dp_0\|_{\infty,Dp_0}$.
The claim of the proposition then follows from (\ref{LIKtrick2}) and
Lemma \ref{LIKgsc}.
\end{proof}

We are now ready to prove Theorem~\ref{thm:bv} and Theorem~\ref{LIKkw2}.
\begin{proof}[Proof of Theorem~\ref{thm:bv}]
Without loss of generality we can set $f$ equal to 
zero outside of $(0,\alpha_1)$.
We use Lemma \ref{LIKclas},
Proposition \ref{LIKcalc}, $\hat p_n, p_0 \in \mathcal U(X_1, \dots, X_n)$ by
Lemma \ref{LIKlow} and a Taylor expansion up to second order to see
 \begin{align*}
& |(\hat P_n -P_n)(f)|  = |D\ell_n(p_0)[\pi_0(f)]+ D^2 \ell(p_0)[\hat p_n-p_0,
\pi_0(f)]| \\
 &\le |D\ell_n(\hat p_n)[\pi_0(f)]| + |(D^2 \ell_n(p_0)- D^2 \ell (p_0))[\hat
p_n-p_0, \pi_0(f)]| \\
&~~~~ + \tfrac{1}{2}|(D^3 \ell_n(\bar p_n)-D^3\ell(\bar p_n))[\hat p_n -p_0,
\hat p_n-p_0, \pi_0(f)]| \\
&~~~~ + \tfrac{1}{2}|D^3\ell(\bar p_n)[\hat p_n -p_0, \hat
p_n-p_0, \pi_0(f)]|,
\end{align*}
where $\bar p_n$ equals, on $[0,X_{(n)}]$, some mean values between $\hat
p_n$ and $p_0$, and $\bar p_n^{-1}$ is zero otherwise by convention. 
Here again $D^3\ell_n(\bar p_n)$ and $D^3\ell(\bar p_n)$ stand for the
corresponding right hand sides in~\eqref{LIKcalcEmp} and~\eqref{LIKcalcEqu}.
The first term is bounded using Proposition~\ref{propBV}, giving the
bound $B n^{-2/3}$ in probability.
We define $h\equiv p_0^{-1}(\hat p_n-p_0)(f-P_0 f)$ on 
$[0,\alpha_1]$ and $h\equiv0$ 
elsewhere so that the second term equals
$|(P_n-P_0)h|$.
With probability arbitrarily close to one we have
$\|h\|_{BV}\lesssim\|f\|_\infty+\|f\|_{BV}
\lesssim\|f\|_\infty+\|Df/Dp_0\|_{\infty,Dp_0}$ and 
$\|h\|_{2,P_0}\lesssim\|f\|_\infty
n^{-1/3}$.
The second term is bounded similarly as in
(\ref{LIKsbd}) above by 
$$ \sup_{h: \|h\|_{BV} \le \tilde CB, \|h\|_{2, P_0} \le \tilde M B
n^{-1/3}}|(P_n-P_0)(h)| =
O_{P_0^\mathbb N}(B n^{-2/3}). $$
The third term is bounded the same way, using $\|\hat p_n 
-p_0\|_{BV}=O_{P_0^\mathbb N}(1)$, and
noting that $\bar p_n$ as a convex combination of $\hat p_n, p_0$ has variation
bounded by a fixed constant on $[0,X_{(n)}]$, so that we can estimate the term
by the supremum of the empirical process over a fixed $BV$-ball, and using again
Lemma \ref{LIKlow} to bound~$\bar p_n$ from below on $[0, X_{(n)}]$. Using the
last fact the fourth term is also seen to be of order $$\|f\|_\infty\|\hat p_n
-p_0\|^2_2 = O_{P_0^\mathbb
N}(Bn^{-2/3})$$ in view of \eqref{LIKgrenrat} completing the
proof the first claim. The second claim follows from the fact that~$\mathcal B$ 
is a bounded set in the space of bounded variation
functions and thus a Donsker class, which follows from Theorem~2.7.5 in 
\cite{vanderVaartWellner1996}.
\end{proof}

\medskip

\begin{proof}[Proof of Theorem~\ref{LIKkw2}]
It is sufficient to prove the result for $1/2<s<1$ since the H\"older spaces 
are nested. 
Let $[a,b]$ be a compact interval.
In order to define Besov 
spaces $B^s_{pq}([a,b])$, $1\le p\le \infty$, 
$1\le q \le\infty$, $0<s<S$, we consider
a boundary corrected Daubechies wavelet basis of regularity $S$ and such that 
$\phi,\psi\in C^S([a,b])$, see \cite{CohenDaubechiesVial1993}.
We define Besov spaces as in 
\citep{GineNickl2015} by the wavelet 
characterisation
\begin{align*}
 B^s_{pq}([a,b])\equiv\left\{
\begin{array}{ll}
 \{f\in L^p([a,b]):\|f\|_{B^s_{p,q}}<\infty\}, & 1\le p<\infty,\\
\{f\in C([a,b]):\|f\|_{B^s_{p,q}}<\infty\}, & p=\infty,
\end{array}
\right.
\end{align*}
with norms given by
\begin{align*}
& \|f\|_{B^s_{pq}([a,b])}\\
&\equiv
\left\{
\begin{array}{l}
 \left(\sum\limits_{k=0}^{2^J-1}|\langle f,\phi_{Jk} 
\rangle|^p\right)^{\frac{1}{p}}+\left(\sum\limits_{l=J}^{\infty}2^{ql(s+\frac1{2
}
-\frac{1}{p}) } \left(\sum\limits_{m=0}^{2^l-1}|\langle f,\psi_{lm} 
\rangle|^p\right)^{\frac{q}{p}}\right)^{\frac1{q}},  p<\infty,\\
 \max\limits_{k}|\langle f,\phi_{Jk} 
\rangle|+\left(\sum\limits_{l=J}^{\infty}2^{ql(s+\frac1{2}) } 
\left(\max\limits_{m}|\langle 
f,\psi_{lm} 
\rangle|\right)^{q}\right)^{\frac1{q}},  p=\infty,
\end{array}
\right.
\end{align*}
where in the case $q=\infty$ the $\ell_q$-sequence norm has to be replaced by 
the supremum norm $\|\cdot\|_\infty$.

Without loss of generality we consider a ball $\mathcal F$ in the H\"older 
space $C^s(S_0)$.
We decompose the functions~$f$ in a ball $\mathcal F$ of $C^s(S_0)$ by using the
projection $\pi_{V_j}(f)$ onto the span of the wavelets up to resolution level
$j$,
\begin{equation}\label{LIKdecomp2}
  \begin{aligned} 
\sup_{f\in\mathcal F}|(\hat P_n
-P_n)(f)| & \le \sup_{f \in \mathcal F}\left|\int_{S_0} (\hat
p_n -p_0) (f-\pi_{V_j}(f))\right|  \\
& ~~+ \sup_{f \in \mathcal F}| (\hat P_n-P_n)(\pi_{V_j}(f))| \\
& ~~+ \sup_{f \in
\mathcal F} |(P_n-P_0)(f-\pi_{V_j}(f))|.
\end{aligned}
\end{equation}
Since $C^s(S_0)=B^s_{\infty
\infty}(S_0)$ for $s \notin \mathbb N$ and since the
$C^1$-norm  is bounded by the $B^1_{\infty 1}$-norm, we have for the wavelet
partial sum $\pi_{V_j}(f)$ of $f \in C^s(S_0)$
using the unified notation
$\psi_{-1,k}=\phi_{l_0,k}$
\begin{align*}
\|\pi_{V_j}(f)\|_{C^1} \lesssim \sum_{l \le j} 2^{3l/2} \max_{k} |\langle f,
\psi_{lk} \rangle| &\lesssim 2^{j(1-s)}\max_{l \le j} 2^{l(s+1/2)} \max_{l \le
j,k}
|\langle f, \psi_{lk} \rangle| \\
&\le 2^{j(1-s)} \|f\|_{B^s_{\infty \infty}}.
\end{align*}
Thus taking $2^{j} \sim n^{1/3}$ we have by Proposition~\ref{propBV}
$$\sup_{f \in \mathcal F}|(\hat P_n -P_n)(\pi_{V_j}(f))| =O_{P_0^\mathbb
N}(n^{-2/3} n^{(1-s)/3}) = o_{P_0^\mathbb
N}(1/\sqrt{n})$$
since $s>1/2$. 
Moreover, by Parseval's identity $$\sup_{f \in \mathcal F}\|\pi_{V_j}(f)-f\|_2=
O(2^{-js}).$$ 
Also, using the $L^2$-convergence rate in~\eqref{LIKgrenrat} and the 
Cauchy--Schwarz
inequality $$\sup_{f \in \mathcal F}\left |\int_0^1 (\hat p_n-p_0)
(f-\pi_{V_j}(f)) \right| =
O_{P_0^\mathbb
N}(n^{-1/3}n^{-s/3}) = o_{P_0^\mathbb
N}(1/\sqrt n)$$ and since the class
$\{f-\pi_{V_j}(f)\}$ is contained in the fixed $s$-H\"older ball $\mathcal F$, 
which is a
$P_0$-Donsker class for $s>1/2$ in view of Corollary~5 in
\cite{NicklPoetscher2007}, and has
envelopes that converge to zero we see that the third term in (\ref{LIKdecomp2})
is also $o_{P_0^\mathbb N}(1/\sqrt n)$ (since the empirical process is tight and
has a degenerate Gaussian limit). The remaining claims follow from the fact that
$\mathcal F$ is a $P_0$-Donsker class. 
\end{proof}

\bibliography{bib}

\end{document}